\documentclass[12pt]{article}

\usepackage{amsthm}

\newtheorem{theorem}{Theorem}[section]

\newtheorem{lemma}[theorem]{Lemma}

\newtheorem{corollary}[theorem]{Corollary}

\newtheorem{proposition}[theorem]{Proposition}

\theoremstyle{remark}
\newtheorem{remark}[theorem]{Remark}
\theoremstyle{definition}
\newtheorem{definition}[theorem]{Definition}

\usepackage{amssymb,amsfonts,amsthm,amsmath,amsopn,amstext,amscd,latexsym,color}
\usepackage[all]{xy}

\usepackage[pdftex]{graphicx}

\xyoption{web}

\newcommand{\bbN}{{\mathbb N}}

\newcommand{\bbZ}{{\mathbb Z}}

\newcommand{\cB}{{\mathcal B}}

\newcommand{\cD}{{\mathcal D}}

\newcommand{\cR}{{\mathcal R}}
\newcommand{\cS}{{\mathcal S}}
\newcommand{\cT}{{\mathcal T}}

\newcommand{\Kk}{{\scriptscriptstyle{K}}}
\newcommand{\Mm}{{\scriptscriptstyle{M}}}

\usepackage[pdftex]{graphicx}

\newtheorem{tau-sigma}[unram]{Lemma}

\newtheorem{general element}{Lemma}[section]

\newtheorem{general congruence}{Proposition}[section]
\theoremstyle{remark}\newtheorem{any unif}[general congruence]{Remark}
\theoremstyle{plain}
\newtheorem{pipower congruence}[general congruence]{Corollary}
\newtheorem{unit congruence}[general congruence]{Corollary}

\newtheorem{ramgeneral element}{Lemma}[section]

\theoremstyle{remark}

\theoremstyle{remark}

\DeclareMathOperator{\Gal}{Gal}
\DeclareMathOperator{\Br}{Br}

\DeclareMathOperator{\Aut}{Aut}
\DeclareMathOperator{\Hom}{Hom}
\DeclareMathOperator{\End}{End}
\DeclareMathOperator{\Tr}{Tr}

\DeclareMathOperator{\Char}{char}

\author{Rachel Newton}
\title{Realising the cup-product of local Tate duality}
\date{}

\begin{document}
\maketitle
\begin{abstract}We present an explicit description, in terms of central simple algebras, of a cup-product map which occurs in the statement of local Tate duality for Galois modules of prime order $p$. Given cocycles $f$ and $g$, we construct a central simple algebra of dimension $p^2$ whose class in the Brauer group gives the cup-product $f\cup g$. This algebra is as small as possible.
\end{abstract}


\section*{Introduction}
Let $K$ be any field and let $M$ be a $G_{\Kk}$-module of prime cardinality $p$, where $p$ is not equal to $\Char(K)$. In this paper, we give an explicit description of the following cup-product which occurs in the statement of local Tate duality:
\begin{equation}
\label{eq:Tate duality}
\cup:H^1(G_{\Kk},M)\times H^1(G_{\Kk},M^{\vee})\longrightarrow H^2(G_{\Kk},\mu_p)\cong \Br(K)[p].
\end{equation}
The main result is Theorem \ref{mainthm} where, given elements $0\neq f\in H^1(G_{\Kk},M)$ and $0\neq g\in H^1(G_{\Kk},M^{\vee})$, we construct a central simple algebra $\cD$ such that
\begin{enumerate}
\item The class of $\cD$ in $\Br(K)$ is the class of the cup-product $f\cup g$.
\item $\dim_K(\cD)=p^2$. Therefore, $\cD$ is a division algebra if and only if \mbox{$f\cup g\neq 0.$}
\end{enumerate}
In the prime order case, the usual construction gives a central simple algebra which can have dimension as large as $p^4(p-1)^4$ in general. 
Our minimisation of the dimension of the central simple algebra makes the cup-product (\ref{eq:Tate duality}) more amenable to explicit computation. 

\section[The Artin-Wedderburn theorem and Tate duality]{The Artin-Wedderburn theorem and local Tate duality}
Let $K$ be a field. We will consider $K$ to be fixed throughout the paper and will use the following notation:\\
\begin{tabular}{ll}
$K_s$ &  a fixed separable closure of $K$\\
$G_{\Kk}$ & the absolute Galois group of $K$, $G_{\Kk}=\Gal(K_s/K)$\\
$M$ & a $G_{\Kk}$-module of cardinality $p$ prime and not divisible by $\Char(K)$\\
$\mu_p$ & the group of $p$th roots of unity in $K_s$\\
$M^{\vee}$ & the Tate dual of $M$, $M^{\vee}=\Hom(M,\mu_p)$.\\
\end{tabular}\\
For elements $f,g,\varphi,\ldots$ of cohomology groups, we often employ the notation $f_0,g_0,\varphi_0, \ldots$ to refer to a choice of representative cocycles.
From now on, we fix $0\neq f\in H^1(G_{\Kk},M)$ and $0\neq g\in H^1(G_{\Kk},M^{\vee})$. In order to compute the cup-product $f\cup g$ as a central simple algebra, we must replace $G_{\Kk}$ with a finite Galois group. The action of $G_{\Kk}$ on $M$ gives a map 
$G_{\Kk}\rightarrow \Aut(M).$ Let $H_{\Mm}$ denote the kernel of this map and consider the inflation-restriction exact sequence
\begin{displaymath}
\xymatrix{
0\ar[r] & H^1(G_{\Kk}/H_{\Mm}, M)\ar[r]^{\quad\textrm{Inf}} & H^1(G_{\Kk},M)\ar[r]^{\textrm{Res}\quad\quad} & H^1(H_{\Mm},M)^{G_{\Kk}/H_{\Mm}}& \\
\ar[r] & H^2(G_{\Kk}/H_{\Mm}, M).\\
}
\end{displaymath}
Observe that $G_{\Kk}/H_{\Mm}$ injects into $\Aut(M)$, which has order $p-1$. Hence, $G_{\Kk}/H_{\Mm}$ has order coprime to $\#M=p$ and $$H^1(G_{\Kk}/H_{\Mm}, M)=H^2(G_{\Kk}/H_{\Mm}, M)=0.$$ Therefore, the restriction map gives an isomorphism 
\[H^1(G_{\Kk},M)\cong H^1(H_{\Mm},M)^{G_{\Kk}/H_{\Mm}}=\Hom_{G_{\Kk}}(H_{\Mm},M).\]
Let $N_f$ denote the kernel of the restriction of $f$ to $H_{\Mm}$. Then the isomorphism above shows that $N_f\triangleleft G_{\Kk}$. Because $f\neq 0$, the injective $G_{\Kk}$-homomorphism $H_{\Mm}/N_f\rightarrow M$ induced by $f$ is also surjective. So $H_{\Mm}/N_f$ has order $p$. In the same way, we define $H_{\Mm^{\vee}}$ and $N_g$. 

\begin{lemma}
If $N_f=N_g$, then $M$ and $M^{\vee}$ are isomorphic as $G_{\Kk}$-modules.
\end{lemma}

\begin{proof}
We have isomorphisms of $G_{\Kk}$-modules $H_{\Mm}/N_f\rightarrow M$ and $H_{\Mm^{\vee}}/N_g\rightarrow M^{\vee}$ induced by $f$ and $g$ respectively. So it suffices to show that $H_{\Mm}/N_f=H_{\Mm^{\vee}}/N_g$. But $H_{\Mm}/N_f$ is the unique Sylow $p$-subgroup of $G_{\Kk}/N_f\cong H_{\Mm}/N_f\rtimes G_{\Kk}/H_{\Mm}$ and $H_{\Mm^{\vee}}/N_g$ is also an order $p$ subgroup of $G_{\Kk}/N_f=G_{\Kk}/N_g$.
\end{proof}

\begin{corollary}
\label{fcupg=0}
If $N_f=N_g$ and $p>2$, then $f\cup g=0$.
\end{corollary}

\begin{proof}
By the lemma above, we know that $M$ and $M^{\vee}$ are isomorphic as $G_{\Kk}$-modules. The cup-product map is anti-symmetric and $p>2$ so anti-symmetric implies alternating. Thus, it is enough to show that $g=nf$ for some $n\in\bbZ$. The restriction map $\textrm{Res}:H^1(G_{\Kk},M)\rightarrow H^1(H_{\Mm},M)=\Hom(H_{\Mm},M)$ is injective, so it suffices to show that $\textrm{Res}(g)=n\textrm{Res}(f)$ for some $n\in\bbZ$. Now $\textrm{Res}(f)$ and $\textrm{Res}(g)$ both have kernel $N_f$, so they both arise from isomorphisms $H_{\Mm}/N_f\rightarrow M$. But $M$ has order $p$, so any two such isomorphisms differ by a scalar multiple.
\end{proof}

Let $N=N_f\cap N_g$. Consider the inflation-restriction exact sequence

\begin{displaymath}
\xymatrix{
0\ar[r] & H^1(G_{\Kk}/N, M)\ar[r]^{\quad\textrm{Inf}} & H^1(G_{\Kk},M)\ar[r]^{\textrm{Res}} & H^1(N,M).
}
\end{displaymath}
By definition of $N$, the element $f$ is in the kernel of restriction to $N$. So $f$ comes from an element of $H^1(G_{\Kk}/N, M)$, which we will also call $f$. Similarly, $g$ comes from an element of $H^1(G_{\Kk}/N, M)$, which we will also call $g$. Now, the properties of the cup-product mean that the following diagram commutes:
\begin{displaymath}
\xymatrix{
H^1(G_{\Kk}, M)\times  H^1(G_{\Kk},M^{\vee})\ar[r]^{\quad \quad\quad\quad\quad \cup}  & H^2(G_{\Kk},\mu_p)\\
H^1(G_{\Kk}/N, M)\ar@<-9ex>[u]^{\textrm{Inf}}\times  H^1(G_{\Kk}/N,M^{\vee})\ar@<7ex>[u]^{\textrm{Inf}}\ar[r]^{\quad\quad \quad\quad\quad\cup} &  H^2(G_{\Kk}/N,\mu_p)\ar[u]^{\textrm{Inf}}\\
}
\end{displaymath}

Therefore, we can reduce to studying the cup-product
\begin{equation}
\label{eq:Tate duality finite}
\cup:H^1(G_{\Kk}/N,M)\times H^1(G_{\Kk}/N,M^{\vee})\longrightarrow H^2(G_{\Kk}/N,\mu_p).
\end{equation}
Let $L=K_s^N$ so that $\Gal(L/K)=G_{\Kk}/N$. Thus, $L/K$ is a finite Galois extension of degree dividing $p^2(p-1)^2$. Note that the action of $G_{\Kk}$ on $M^{\vee}=\Hom(M,\mu_p)$ is given by $(s\cdot\phi)(m)=s\cdot\phi(s^{-1}\cdot m)$ for all $s\in G_{\Kk}$ and all $m\in M$. Hence, $\mu_p$ is fixed by all elements in $H_{\Mm}\cap H_{\Mm^{\vee}}$, so $\mu_p\subset L^*$. We have the following commutative diagram:
\begin{displaymath}
\xymatrix{
 H^2(G_{\Kk},\mu_p)\ar@{^{(}->}[r]& H^2(G_{\Kk}, K_s^*)\ar[r]^{\quad\cong} & \Br(K)\\
 H^2(\Gal(L/K),\mu_p)\ar[u]^{\textrm{Inf}}\ar[r] & H^2(\Gal(L/K), L^*)\ar[r]^{\quad\quad\cong} & \Br(L/K)\ar@{^{(}->}[u]\\
}
\end{displaymath}
where $\Br(L/K)$ denotes the subgroup of $\Br(K)$ consisting of the classes of central simple algebras over $K$ which are split by $L/K$. The isomorphism $H^2(\Gal(L/K), L^*)\rightarrow \Br(L/K)$ is induced by the map sending a $2$-cocycle $\vartheta$ to the central simple algebra $A_{\vartheta}$ as defined below.
\begin{definition}
Let $L/K$ be a finite Galois extension and let $\vartheta\in Z^2(\Gal(L/K),L^*)$ be a $2$-cocycle. Define the $K$-algebra $A_{\vartheta}$ to be the left $L$-vector space with basis $\{e_s\}_{s\in\Gal(L/K)}$ and multiplication given by
\begin{eqnarray*}
e_s x=s(x)e_s\ \quad\forall\ s\in \Gal(L/K), \ \quad\forall\ x\in L\\
e_se_t=\vartheta(s,t)e_{st}\ \quad\forall\ s,t\in\Gal(L/K).\\
\end{eqnarray*}
$A_{\vartheta}$ is a central simple algebra of dimension $[L:K]^2$ over $K$. See, for example, \cite{Reiner}, where this is Theorem 29.12. 
\end{definition}
\begin{definition}
Let $\varphi=f\cup g$. Fix representative cocycles $f_0, g_0$ for $f,g$ respectively. The formula given in the remark at the end of \S2.4 of \cite{Pillons} tells us that a representative $2$-cocycle for $\varphi$ is
$\varphi_0:\Gal(L/K)\times \Gal(L/K)\rightarrow \mu_p$, given by
\begin{equation}
\label{phi}
\varphi_0(s,t)=(s\cdot g_0(t))(f_0(s)).
\end{equation}
\end{definition}
\begin{lemma}
\label{fcupg for 2}
If $N_f = N_g$ and $p=2$, then $f\cup g$ corresponds to a quaternion algebra generated by $x,y$ such that $K(x)\cong K_s^{N_f}$, $x^2\in K^*$, $y^2=-1$ and $yx=-xy$. Consequently, $f\cup g=0$ if and only if $-1\in N_{K_s^{N_f}/K}(K_s^{N_f})$.
\end{lemma}

\begin{proof}
This follows from the explicit construction of a central simple algebra given above. By \cite{Jacobson}, Theorem 8.14, the quaternion algebra $A_{\varphi_0}$ is a division ring if and only if \mbox{$y^2\notin N_{K(x)/K}(K(x)).$} 
\end{proof}

Having dealt with the case $N_f=N_g$ for all $p$, henceforth we assume that $N_f\neq N_g$. 
Below, we state the main result which will be proved in this paper. 

\begin{theorem}
\label{mainthm}
Write $K_s^{\ker(f_0)}=K(\alpha)$, $K_s^{\ker(g_0)}=K(\beta)$ with $\Tr_{K(\alpha)/K}(\alpha)=0=\Tr_{K(\beta)/K}(\beta)$. Let $\sigma\in G_{\Kk}$ be such that $\sigma$ fixes the normal closure of $K(\beta,\mu_p)$ and $\sigma(\alpha)\neq\alpha$. Likewise, let $\rho\in G_{\Kk}$ act trivially on the normal closure of $K(\alpha,\mu_p)$ but non-trivially on $\beta$. Let $\zeta=(g_0(\rho))(f_0(\sigma))\in\mu_p$. Let $h_{ij}=\sum_{\ell=0}^{p-1}{\zeta^{j\ell}\sigma^{\ell}(\alpha^i)}$. Write $\rho^{j}(\beta)=\sum_{i=0}^{p-1}{m_{i j}\beta^{i}}$ for $m_{i j}\in K_s^{H_{M^{\vee}}}$. 
Let $\cD$ be the left $K(\beta)$-vector space with basis $\{z^j\}_{0\leq j\leq p-1}$, where $z$ satisfies the same minimal polynomial over $K$ as $\alpha$, with multiplication
$$z\beta=\sum_{i,j=0}^{p-1}{c_{ij}\beta^iz^j}$$
where the matrix $(c_{ij})_{i,j}=(h_{1j} m_{ij})_{i,j}(h_{ij})_{i,j}^{-1}$. Then $\cD$ is a central simple algebra of dimension $p^2$ over $K$ which gives the class of $f\cup g$ in $\Br(K)$.
\end{theorem}
\begin{corollary}
Suppose that $p=2$. Then $f\cup g$ corresponds to a quaternion algebra over $K$, generated by two elements $x$ and $y$ such that $K(x)\cong K_s^{\ker(g_0)}$ and $K(y)\cong K_s^{\ker(f_0)}$, with $x^2,y^2\in K$ and $yx=-xy$. Consequently, $f\cup g$ is trivial if and only if $x^2\in N_{K(y)/K}(K(y))$, if and only if $y^2\in N_{K(x)/K}(K(x))$.
\end{corollary}

\begin{proof}
This follows immediately from the Theorem \ref{mainthm}. The quaternion algebra is a division ring if and only if $x^2\notin N_{K(y)/K}(K(y))$, if and only if $y^2\notin N_{K(x)/K}(K(x))$ by \cite{Jacobson}, Theorem 8.14.
\end{proof}

The algebra $A_{\varphi_0}$ has dimension at most $p^4(p-1)^4$ over $K$. The Artin-Wedderburn Theorem tells us that $A_{\varphi_0}\cong M_n(D)$ for some $n\in \bbN$ and some division algebra $D$. The quantity $\sqrt{\dim_K(D)}$ is called the \textit{index} of $A_{\varphi_0}$.

\begin{lemma}
$K_s^{\ker(f_0)}/K$ and $K_s^{\ker(g_0)}/K$ are degree $p$ subextensions of $L$ which split $A_{\varphi_0}$.
\end{lemma}

\begin{proof}
First, note that $\ker(f_0)$ is a subgroup of $G_{\Kk}$ because $f_0$ is a $1$-cocycle. Also, $f_0$ defines an injection from the left cosets of $\ker(f_0)$ in $G_{\Kk}$ to $M$. This injection is also a surjection because the restriction of $f$ to $H_{\Mm}$ surjects onto $M$. Thus, $K_s^{\ker(f_0)}/K$ is a degree $p$ extension. Since $N\subset N_f\subset \ker(f_0)$, we have $K_s^{\ker(f_0)}\subset L$. The following diagram commutes:
\begin{equation}
\label{res is tensor}
\xymatrix{
 H^2(\Gal(L/K), L^*)\ar[d]^{\textrm{Res}}\ar[r]^{\quad\quad\cong} & \Br(L/K)\ar[d]\\
 H^2(\Gal(L/K_s^{\ker(f_0)}), L^*)\ar[r]^{\quad\quad\cong} & \Br\bigl(L/K_s^{\ker(f_0)}\bigr)\\
 }
\end{equation}
where the map $\Br(L/K)\rightarrow\Br\bigl(L/K_s^{\ker(f_0)}\bigr)$ is induced by $A\mapsto A\otimes_K K_s^{\ker(f_0)}$. The restriction of $f$ to $\ker(f_0)$ is trivial in $H^1(\ker(f_0), M)$, and the cup-product commutes with the restriction homomorphism. So we have
\[\textrm{Res}(f\cup g)=\textrm{Res}(f)\cup\textrm{Res}(g)=0\cup\textrm{Res}(g)=0.\]
Therefore, diagram (\ref{res is tensor}) shows that $A_{\varphi_0}\otimes_K K_s^{\ker(f_0)}$ is trivial in $\Br(L/K_s^{\ker(f_0)})$. In other words, $K_s^{\ker(f_0)}$ splits $A_{\varphi_0}$. The argument for $K_s^{\ker(g_0)}$ is analogous.
\end{proof}

\begin{remark}
If $f_0$ is modified by a coboundary, the subgroup $\ker(f_0)$ is conjugated by an element of $G_{\Kk}$. Thus, the embedding of $K_s^{\ker(f_0)}$ in $L$ is changed. But the $K$-isomorphism class of the field $K_s^{\ker(f_0)}$ only depends on $f$.
\end{remark}

\begin{corollary}
Suppose that the class of $A_{\varphi_0}$ in $\Br(K)$ is non-trivial. Then $A_{\varphi_0}$ is isomorphic to $M_n(D)$, where $D$ is a central division algebra over $K$ of dimension $p^2$ and $n=p^{-1}[L:K]$. Thus, the index of $A_{\varphi_0}$ is equal to its period, $p$. Moreover, $K_s^{\ker(f_0)}$ and $K_s^{\ker(g_0)}$ embed into $D$ as maximal commutative subalgebras.
\end{corollary}

\begin{proof}
The index of $A_{\varphi_0}$ is the greatest common divisor of the degrees of finite separable extensions which split $A_{\varphi_0}$. The extension $K_s^{\ker(f_0)}/K$ splits $A_{\varphi_0}$. Since $K_s^{\ker(f_0)}/K$ has degree $p$, the index of $A_{\varphi}$ is $p$. Consequently, $A_{\varphi_0}\cong M_n(D)$, where $D$ is a central division algebra of dimension $p^2$ over $K$, and $D$ has a maximal commutative subalgebra isomorphic to $K_s^{\ker(f_0)}$. Likewise, $K_s^{\ker(g_0)}$ also embeds into $D$ as a maximal commutative subalgebra. Moreover, $A_{\varphi_0}$ has $K$-dimension $[L:K]^2=n^2[D:K]=n^2p^2$. Therefore, $n=p^{-1}[L:K]$.
\end{proof}

We want to compute $D$ explicitly and relate its generators to the splitting fields $K_s^{\ker(f_0)}$ and $K_s^{\ker(g_0)}$. The proof of the Artin-Wedderburn Theorem shows that $D\cong \End_{A_{\varphi_0}}(S)^{\textrm{opp}}$ for any minimal left ideal $S$. 
The same proof also shows that a left ideal $I$ of $A_{\varphi_0}\cong M_n(D)$ is minimal if and only if 
\[\dim_K(I)=n[D:K].\]
\begin{definition} Let $\theta= \sum_{t\in \Gal(L/K_s^{\ker(g_0)})}{e_t}$ and let $\cS$ be the left ideal of $A_{\varphi_0}$ generated by $\theta$.
\end{definition}

\begin{proposition}
\label{dim S}
We have $\cS=\{xe_s\theta \mid x\in L, s\in R\}$, where $R$ is a set of left coset representatives for $\Gal(L/K_s^{\ker(g_0)})$ in $\Gal(L/K)$. Moreover, the dimension of $\cS$ as a $K$-vector space satisfies the following equality:
$$\dim_K(\cS)=[K_s^{\ker(g_0)}:K][L:K]=p[L:K].$$ 
\end{proposition}

\begin{proof}
The elements $\{e_s\theta\}_{s\in\Gal(L/K)}$ span the left $L$-vector space \mbox{$\cS=A_{\varphi_0}\theta$}. For any $s\in \Gal(L/K)$, we have 
\begin{eqnarray*}
e_s\theta = 
\sum_{t\in \Gal(L/K_s^{\ker(g_0)})}{e_se_t}
= \sum_{t\in \Gal(L/K_s^{\ker(g_0)})}{\varphi_0(s,t)e_{st}} = \sum_{t\in \Gal(L/K_s^{\ker(g_0)})}{e_{st}}
\end{eqnarray*}
where the last equality holds because $\varphi_0(s,t)=1$ for all $t\in \Gal(L/K_s^{\ker(g_0)})$, by definition of $\varphi_0$. In particular, if $s\in \Gal(L/K_s^{\ker(g_0)})$, we have $e_s\theta=\theta$. So if $R$ is a set of left coset representatives for $\Gal(L/K_s^{\ker(g_0)})$ in $\Gal(L/K)$, the elements $\{e_s\theta\}_{s\in R}$ span the left $L$-vector space $\cS$. In fact, these elements form a left $L$-basis for $\cS$. To show linear independence, suppose that 
$$\sum_{s\in R}{x_s e_s \theta}=0$$
for some coefficients $x_{s}\in L$. Then we have
\begin{eqnarray*}
0&=&\sum_{s\in R}{x_s e_s \theta} =\sum_{s\in R}{x_s \sum_{t\in \Gal(L/K_s^{\ker(g_0)})}{e_{st}}} =\sum_{r\in \Gal(L/K)}{x_r e_r}.\\
\end{eqnarray*}
But the elements $\{e_r\}_{r\in\Gal(L/K)}$ form a left $L$-basis for $A_{\varphi_0}$. Therefore, we must have $x_{s}=0$ for all $s\in R$. Hence, the elements $\{e_s\theta \}_{s\in R}$ form a left $L$-basis for $\cS$, with $|R|$ distinct elements. The cardinality of $R$ is 
\begin{eqnarray*}
\frac{|\Gal(L/K)|}{|\Gal(L/K_s^{\ker(g_0)})|}&=&[K_s^{\ker(g_0)}:K]=p,\\
\end{eqnarray*}
whereby the dimension of $\cS$ as a $K$-vector space is $p[L:K]$, as required.
\end{proof}

\begin{corollary}
\label{S min}
If the class of $A_{\varphi_0}$ in $\Br(K)$ is non-trivial, then $\cS$ is a minimal left ideal of $A_{\varphi_0}$.
\end{corollary}

\begin{proof}
The proof of the Artin-Wedderburn Theorem shows that a left ideal in $A_{\varphi_0}$ is minimal if and only if its dimension over $K$ is equal to $n[D:K]$, where $A_{\varphi_0}\cong M_n(D)$. If the class of $A_{\varphi_0}$ in $\Br(K)$ is non-trivial, then we have
\[[L:K]^2=\dim_K(A_{\varphi_0})=n^2[D:K]=n^2p^2.\]
Thus, a left ideal in $A_{\varphi_0}$ is minimal if and only if its dimension over $K$ is equal to $np^2=p[L:K]$. 
\end{proof}

\begin{corollary}
If $D\neq K$, then $D\cong \End_{A_{\varphi_0}}(\cS)^{\textrm{opp}}$.
\end{corollary}

\begin{proof}
By definition, the class of $A_{\varphi_0}$ in $\Br(K)$ is trivial if and only if $D=K$.
The proof of the Artin-Wedderburn Theorem shows that $D\cong \End_{A_{\varphi_0}}(S)^{\textrm{opp}}$ for any minimal left ideal $S$ of $A_{\varphi_0}$.
Thus, the result follows from Corollary \ref{S min}.
\end{proof}

\begin{remark}
If the class of $A_{\varphi_0}$ in $\Br(K)$ is trivial, then $\cS$ is no longer a minimal left ideal of $A_{\varphi_0}$. But $\End_{A_{\varphi_0}}(\cS)^{\textrm{opp}}$ is still a central simple algebra over $K$ of dimension $p^2$ with the same class in $\Br(K)$ as $A_{\varphi_0}$. We will prove that $\cD=\End_{A_{\varphi_0}}(\cS)^{\textrm{opp}}$ is as described in Theorem \ref{mainthm}.

\end{remark}

\section{Computing the endomorphism ring}

Let $R$ be a set of left coset representatives for $\Gal(L/K_s^{\ker(g_0)})$ in $\Gal(L/K)$ and let $B=\{xe_s \mid x\in L, s\in R\}$. Proposition \ref{dim S} tells us that $\cS=A_{\varphi_0}\theta=B\theta$. We would like $B$ to be a subalgebra of $A_{\varphi_0}$, so we want to choose $R$ so that it is a subgroup of $\Gal(L/K)$.

\begin{lemma}
\label{R gen by rho}
Let $\rho\in H_{\Mm^{\vee}}/N$
 be such that its image generates $H_{\Mm^{\vee}}/N_g$. Then $R=\{\rho^i\}_{0\leq i\leq p-1}$ is a set of left coset representatives for $\Gal(L/K_s^{\ker(g_0)})$ in $\Gal(L/K)$.
\end{lemma}

\begin{proof}
We have $\#R=[K_s^{\ker(g_0)}:K]=p$. Thus, it is enough to show that $\rho^r\in \Gal(L/K_s^{\ker(g_0)})=\ker(g_0)/N$ if and only if $p$ divides $r$.
But $\rho\in H_{\Mm^{\vee}}/N$ and $N_g/N=H_{\Mm^{\vee}}/N\cap \ker(g_0)/N$. Hence, $\rho^r\in \ker(g_0)/N$ if and only if $\rho^r\in N_g/N$. But the image of $\rho$ generates $H_{\Mm^{\vee}}/N_g$ and $[H_{\Mm^{\vee}}:N_g]=p$, so $\rho^r\in N_g/N$ if and only if $p$ divides $r$.
\end{proof}

From now on, we fix $R=\{\rho^i\}_{0\leq i\leq p-1}$, so $B$ is a subalgebra of $A_{\varphi_0}$. We want to compute $\End_{A_{\varphi_0}}(\cS)^{\textrm{opp}}$. We know that $\cS$ is a principal left ideal generated by $\theta$, so any $\chi\in\End_{A_{\varphi_0}}(\cS)$ is completely determined by $\chi(\theta)$. Since $\chi(\theta)\in \cS=B\theta$, we have $\chi(\theta)=b\theta$ for some $b\in B$. The question is, which $b$ can occur? In other words, for which $b\in B$ does $\chi:\theta\mapsto b\theta$
extend to a well-defined element of $\End_{A_{\varphi_0}}(\cS)$? The extension of $\chi$ to the whole of $\cS$ is given by 
$$\chi(c\theta)=c\chi(\theta)\ \quad\forall\ c\in B.$$
This is well-defined because any element of $\cS$ can be written as $c\theta$ for a \emph{unique} $c\in B$. But it may not be an $A_{\varphi_0}$-endomorphism.
We see that $\chi$ gives a well-defined element of $\End_{A_{\varphi_0}}(\cS)$ if and only if
\begin{equation*}
\chi(a\theta)=a\chi(\theta)=ab\theta\ \quad  \quad\forall \ a\in A_{\varphi_0}.
\end{equation*}

The point is, when we allow multiplication by the whole of $A_{\varphi_0}$ (rather than just the subalgebra $B$), it is possible to have $a_1\theta=a_2\theta$
with $a_1,a_2\in A_{\varphi_0}$ and $a_1\neq a_2$. For $\chi$ to give a well-defined element of $\End_{A_{\varphi_0}}(\cS)$, we would also need $a_1b\theta=a_2b\theta$
in this case. Equivalently, $\chi$ extends to a well-defined element of $\End_{A_{\varphi_0}}(\cS)$ if and only if $$ab\theta=0 \ \textrm{for all}\ a\in A_{\varphi_0} \ \textrm{such that}\ a\theta=0.$$ Clearly, it suffices for $b$ to commute with $\theta=\sum_{t\in \Gal(L/K_s^{\ker(g_0)})}{e_t}$. Hence, it suffices for $b$ to commute with $e_t$ for every $t\in\Gal(L/K_s^{\ker(g_0)})$.
The multiplication on $\End_{A_{\varphi_0}}(\cS)$ is the opposite of the multiplication on $B$ inherited from $A_{\varphi_0}$. Therefore, we can view $\End_{A_{\varphi_0}}(\cS)^{\textrm{opp}}$ as a subalgebra of $B$. We will make this identification from now on. Thus, we have
\begin{equation}
\label{eq:inclusion}
B\supset\End_{A_{\varphi_0}}(\cS)^{\textrm{opp}}\supset \{b\in B\bigm| e_t b=b e_t\ \quad\forall t\in\Gal(L/K_s^{\ker(g_0)})\}.
\end{equation}

\begin{remark}
In fact, a careful analysis of the left-annihilator of $\theta$ may be used to show that the rightmost inclusion is an equality. We omit the details of this rather involved calculation and instead demonstrate the equality simply by finding enough elements in the right-hand side and comparing dimensions.
\end{remark}
The rightmost inclusion in \eqref{eq:inclusion} leads us to ask the following question. Which elements of $B$ commute with $e_t$ for all $t\in\Gal(L/K_s^{\ker(g_0)})$? Recall that $$B=\{xe_{\rho^i}\mid x\in L,\ 0\leq i\leq p-1\},$$ where $\rho\in H_{\Mm^{\vee}}/N\leq G_{\Kk}/N=\Gal(L/K)$ is such that its image generates $H_{\Mm^{\vee}}/N_g$. Therefore, there is an obvious subalgebra of $B$ whose elements commute with $e_t$ for every $t\in\Gal(L/K_s^{\ker(g_0)})$; namely the field $K_s^{\ker(g_0)}$. This is by definition of the multiplication in $A_{\varphi_0}$; recall that
$$e_s x=s(x)e_s\ \quad\forall s\in \Gal(L/K), \ \quad\forall x\in L.$$
Thus, $x\in L$ commutes with $e_s$ if and only if $s(x)=x$. 



\begin{lemma}
\label{generators for D}
$\End_{A_{\varphi_0}}(\cS)^{\textrm{opp}}$ is generated as a $K$-algebra by the elements of $K_s^{\ker(g_0)}$ and any element $d\in \End_{A_{\varphi_0}}(\cS)^{\textrm{opp}}\setminus K_s^{\ker(g_0)}$.
\end{lemma}

\begin{proof}
We know that the algebra $\End_{A_{\varphi_0}}(\cS)^{\textrm{opp}}$ has dimension $p^2$ over $K$. Let $T=\langle K_s^{\ker(g_0)}, d\rangle$ be the subalgebra of $\End_{A_{\varphi_0}}(\cS)^{\textrm{opp}}$ generated over $K$ by $K_s^{\ker(g_0)}$ and $d$, where $d\in \End_{A_{\varphi_0}}(\cS)^{\textrm{opp}}\setminus K_s^{\ker(g_0)}$. Then, 
$$K\subset K_s^{\ker(g_0)}\subsetneq T\subset \End_{A_{\varphi_0}}(\cS)^{\textrm{opp}}.$$
First, suppose that $\End_{A_{\varphi_0}}(\cS)^{\textrm{opp}}$ is a division ring. Then $T$ is also a division ring and we can view $\End_{A_{\varphi_0}}(\cS)^{\textrm{opp}}$ as a left $T$-vector space. We have
\[p^2=\dim_K\End_{A_{\varphi_0}}(\cS)^{\textrm{opp}}=(\dim_T\End_{A_{\varphi_0}}(\cS)^{\textrm{opp}})(\dim_KT).\]
But $\dim_KT>[K_s^{\ker(g_0)}:K]=p$, whereby $\dim_KT=p^2$ and therefore $\End_{A_{\varphi_0}}(\cS)^{\textrm{opp}}=T.$ 

Now suppose that $\End_{A_{\varphi_0}}(\cS)^{\textrm{opp}}$ is not a division ring. Since $\End_{A_{\varphi_0}}(\cS)^{\textrm{opp}}$ is a central simple algebra of dimension $p^2$ over $K$, the Artin-Wedderburn Theorem tells us that $\End_{A_{\varphi_0}}(\cS)^{\textrm{opp}} \cong M_p(K)$. In other words, $\End_{A_{\varphi_0}}(\cS)^{\textrm{opp}}$ is isomorphic to $\End_K(V)$, where $V$ is a $K$-vector space of dimension $p$. Note that $V$ is a faithful $T$-module. Moreover, 
$$\dim_{K_s^{\ker(g_0)}}{V}=\frac{\dim_K{V}}{[K_s^{\ker(g_0)}:K]}=1.$$
Therefore, $V$ is a simple $K_s^{\ker(g_0)}$-module, and hence a simple $T$-module.
So $T$ has a non-zero faithful simple module, whereby the Jacobson radical of $T$ is zero. Therefore, $T$ is a semisimple $K$-algebra, since $T$ is finite-dimensional over $K$. Now the Artin-Wedderburn Theorem tells us that $T\cong M_m(E)$ for some division ring $E$ over $K$ and some $m\in\bbN$. Furthermore, any nonzero simple module for $M_m(E)$ is isomorphic to the left ideal $I$ of $M_m(E)$ consisting of matrices with all entries zero except in the first column. In particular,
$$p=\dim_K{V}=\dim_K{I}=m[E:K].$$
If $m=1$ and $[E:K]=p$ then $T\cong E$ and we get a contradiction because $K_s^{\ker(g_0)}$ is a proper subalgebra of $T$ of dimension $p$ over $K$. Therefore, we must have $m=p$ and $E=K$, whereby $T\cong M_p(K)$. So $T=\End_{A_{\varphi_0}}(\cS)^{\textrm{opp}}$, as required.
\end{proof}



\begin{proposition}
\label{L_f^tau in D}
$\End_{A_{\varphi_0}}(\cS)^{\textrm{opp}}$ contains a maximal commutative subalgebra isomorphic to $K_s^{\ker(f_0)}$.
\end{proposition}

\begin{proof}
Let $\cT=A_{\varphi_0}\vartheta$ with $\vartheta=\sum_{t\in \Gal(L/K_s^{\ker(f_0)})}{e_t}.$ A similar argument to that of Proposition \ref{dim S} shows that $\dim_K(\cT)=p[L:K]=\dim_K(\cS).$ $A_{\varphi_0}$ is a central simple algebra, so any two $A_{\varphi_0}$-modules with the same finite dimension are isomorphic. Hence, $\cT$ is isomorphic to $\cS$ as an $A_{\varphi_0}$-module. 
In analogy with Lemma \ref{R gen by rho}, choose $\cR=\{\sigma^i\}_{0\leq i\leq p-1}$ where $\sigma\in H_{\Mm}/N$ is such that its image generates $H_{\Mm}/N_f$. Replacing $\cS$ by $\cT$ and following the arguments leading up to Lemma \ref{generators for D}, we find $K_s^{\ker(f_0)}$ as a subalgebra of $\End_{A_{\varphi_0}}(\cT)^{\textrm{opp}}\cong \End_{A_{\varphi_0}}(\cS)^{\textrm{opp}}$.
\end{proof}

Previously, we found $K_s^{\ker(g_0)}$ as a maximal commutative subalgebra of $\End_{A_{\varphi_0}}(\cS)^{\textrm{opp}}$, because the elements of $K_s^{\ker(g_0)}$ commute with $e_t$ for all $t\in \Gal(L/K_s^{\ker(g_0)})$. In view of Proposition \ref{L_f^tau in D} above, we see that $\End_{A_{\varphi_0}}(\cS)^{\textrm{opp}}$ contains a maximal commutative subalgebra isomorphic to $K_s^{\ker(f_0)}$. Therefore, if these two subalgebras are distinct, then together they generate the whole of $\End_{A_{\varphi_0}}(\cS)^{\textrm{opp}}$. In this case, in Lemma \ref{generators for D} we could choose $d\in \End_{A_{\varphi_0}}(\cS)^{\textrm{opp}}\setminus K_s^{\ker(g_0)}$ such that $K(d)\cong K_s^{\ker(f_0)}$.

\section{Finding generators}
\label{generators}
In light of \eqref{eq:inclusion} and Lemma \ref{generators for D}, we seek an element $d\in B\setminus K_s^{\ker(g_0)}$ such that $d$ commutes with $e_t$ for all $t\in\Gal(L/K_s^{\ker(g_0)})$. Recall that
$$B=\{xe_{\rho^i}\mid x\in L,\ 0\leq i\leq p-1\}\subset A_{\varphi_0},$$ where $\rho\in H_{\Mm^{\vee}}/N\leq G_{\Kk}/N=\Gal(L/K)$ is such that its image generates $H_{\Mm^{\vee}}/N_g$. Thus, we can write $d$ in the form
\begin{equation}
\label{eq:d}
d=\sum_{i=0}^{p-1}{a_i e_{\rho^i}} \ \textrm{with}\ a_i\in L.
\end{equation}
We want to find suitable coefficients $a_i$. We will determine the precise conditions on the $a_i$ which must be satisfied if $d$ is to commute with $e_t$ for all $t\in\Gal(L/K_s^{\ker(g_0)})$. 

\begin{lemma}
\label{quotients}
We have $N_fN_g/N=H_{\Mm}H_{\Mm^{\vee}}/N$ and therefore
$$\frac{H_{\Mm}\cap N_g}{N}\cong \frac{H_{\Mm}}{N_f}$$
and $$\frac{H_{\Mm^{\vee}}\cap N_f}{N}\cong \frac{H_{\Mm^{\vee}}}{N_g}.$$
\end{lemma}

\begin{proof}
Clearly, $N_fN_g/N\leq H_{\Mm}H_{\Mm^{\vee}}/N$, so it remains to show the reverse inclusion. We will show that $H_{\Mm}/N\leq N_fN_g/N$; the argument for $H_{\Mm^{\vee}}/N$ is identical. Observe that 
$G_{\Kk}/N_f\cong H_{\Mm}/N_f\rtimes G_{\Kk}/H_{\Mm}$,
where $H_{\Mm}/N_f \cong M$ has order $p$ and $G_{\Kk}/H_{\Mm}\hookrightarrow \Aut(M)$ has order prime to $p$. Thus, any non-trivial normal subgroup of $G_{\Kk}/N_f$ contains $H_{\Mm}/N_f$. Since $N_g\triangleleft G_{\Kk}$, the subgroup $N_fN_g/N_f$ is normal in $G_{\Kk}/N_f$ and $N_fN_g/N_f$ is non-trivial since we are assuming that $N_f\neq N_g$. Therefore, $H_{\Mm}/N_f\leq N_fN_g/N_f$ and hence $H_{\Mm}/N\leq N_fN_g/N$, as required. The last part follows by observing that
\[\frac{H_{\Mm}\cap N_g}{N}\cong \frac{N_f(H_{\Mm}\cap N_g)}{N_f}=\frac{N_fN_g}{N_f}\cap\frac{H_{\Mm}}{N_f}\]
and noting that $N_fN_g/N=H_{\Mm}H_{\Mm^{\vee}}/N$ implies that $N_fN_g/N_f=H_{\Mm}H_{\Mm^{\vee}}/N_f$.
\end{proof}

\begin{lemma}
\label{semidirect}
We have 
$$\Gal(L/K_s^{\ker(g_0)})=\frac{\ker(g_0)}{N}\cong \frac{H_{\Mm}\cap N_g}{N}\rtimes\frac{\ker(g_0)\cap\ker(f_0)}{N}.$$
\end{lemma}

\begin{proof}
Clearly, $\frac{H_{\Mm}\cap N_g}{N}\cap\frac{\ker(g_0)\cap\ker(f_0)}{N}=0.$ It remains to show that $$\Bigl(\frac{H_{\Mm}\cap N_g}{N}\Bigr)\Bigl(\frac{\ker(g_0)\cap\ker(f_0)}{N}\Bigr)=\frac{\ker(g_0)}{N}.$$ Lemma \ref{quotients} shows that $\frac{H_{\Mm}\cap N_g}{N}\cong \frac{H_{\Mm}}{N_f}$. The cocycle $f_0$ gives an isomorphism $\frac{H_{\Mm}}{N_f}\rightarrow M$. So, if $s\in \ker(g_0)$, there exists some $h\in H_{\Mm}\cap N_g$ such that $f_0(h)=f_0(s)$. But then $s=hh^{-1}s$ and $h^{-1}s\in\ker(g_0)\cap\ker(f_0)$.
\end{proof}

We require that
\begin{equation*}
\sum_{i=0}^{p-1}{a_i e_{\rho^i}}=d
=e_s d e_s^{-1}
= \sum_{i=0}^{p-1}{s(a_i) e_se_{\rho^i}e_s^{-1}}.
\end{equation*}
for all $s\in \Gal(L/K_s^{\ker(g_0)})=\ker(g_0)/N$. Lemma \ref{semidirect} allows us to look separately at conjugation by elements in $(H_{\Mm}\cap N_g)/N$ and $(\ker(f_0)\cap\ker(g_0))/N$. First, we look at conjugation by $e_t$ for $t\in (\ker(f_0)\cap\ker(g_0))/N$.


\begin{lemma}
\label{conjbykers}
For all $t\in (\ker(f_0)\cap\ker(g_0))/N$, we have $e_te_{\rho^i}e_t^{-1}=e_{t\rho^i t^{-1}}.$
\end{lemma}

\begin{proof}
If either $s$ or $t$ is in $(\ker(f_0)\cap\ker(g_0))/N$, then \eqref{phi} gives
 $$\varphi_0(s,t)=(s\cdot g_0(t))(f_0(s))=1$$
  and hence $e_se_t=e_{st}.$
Thus, for all $t\in (\ker(f_0)\cap\ker(g_0))/N$, we have $e_t^{-1}=e_{t^{-1}}$ and
$e_te_{\rho^i}e_t^{-1}=e_te_{\rho^i}e_{t^{-1}}=e_{t\rho^it^{-1}}.$
\end{proof}


Lemma \ref{quotients} shows that $(H_{\Mm}\cap N_g)/N\cong H_{\Mm}/N_f$ is a cyclic group of order $p$. Let $\sigma$ be a generator of $(H_{\Mm}\cap N_g)/N$. Lemma \ref{quotients} allows us to assume that $\rho\in (N_f\cap H_{\Mm^{\vee}})/N$, which we will do from now on. In particular, $\sigma$ and $\rho$ act trivially on both $M$ and $M^{\vee}$. Now we consider conjugation by $e_{\sigma}$.

\begin{lemma}
\label{conjbysigma}
We have $e_{\sigma}e_{\rho^i}e_{\sigma}^{-1}=\zeta^i e_{\rho^i}$,
where $\zeta=g(\rho)(f(\sigma))\in K_s$ is a primitive 
$p$th root of unity.
\end{lemma}

\begin{proof}
Recall that $\sigma\in (H_{\Mm}\cap N_g)/N$, so $g(\sigma)=0$. Hence, \eqref{phi} gives
$$\varphi_0(t,\sigma^i)=1\ \quad\forall t\in G_{\Kk}/N, \forall i \in \bbZ.$$
In particular, $e_{\sigma}^{-1}=e_{\sigma^{-1}}$ and we have
\[e_{\sigma}e_{\rho^i}e_{\sigma}^{-1}=e_{\sigma}e_{\rho^i}e_{\sigma^{-1}} = e_{\sigma}\varphi_0(\rho^i, \sigma^{-1})e_{\rho^i\sigma^{-1}} = e_{\sigma}e_{\rho^i\sigma^{-1}} 
=\varphi_0(\sigma, \rho^i\sigma^{-1})e_{\rho^i}.\]
The last line holds because $\sigma$ and $\rho$ commute in $G_{\Kk}/N$, since their commutator is in the intersection of the normal subgroups $(N_g\cap H_{\Mm})/N$ and $(N_f\cap H_{\Mm^{\vee}})/N$, which is trivial. Now, 
\begin{displaymath}
\begin{array}{llll}
\varphi_0(\sigma, \rho^i\sigma^{-1})&=& (\sigma\cdot g_0(\rho^i\sigma^{-1}))(f_0(\sigma))&\\
&=& (\sigma\cdot g_0(\rho^i))(f_0(\sigma))&\textrm{since $g_0(\sigma^{-1})=0$}\\
&=&g(\rho^i)(f(\sigma)) &\textrm{since $\sigma$ acts trivially on $M^{\vee}$}\\
&=&(g(\rho)(f(\sigma)))^i &\textrm{since $g$ gives a homomorphism on $H_{\Mm^{\vee}}$}.\\
\end{array}
\end{displaymath}
Therefore, it suffices to show that $\zeta=g(\rho)(f(\sigma))$ is a primitive $p$th root of unity. We know that $f$ induces an isomorphism $H_{\Mm}/N_f\tilde{\longrightarrow} M$ and $f(\sigma)$ generates $M$ as an abelian group. Likewise, $g$ induces an isomorphism $H_{\Mm^{\vee}}/N_g\tilde{\longrightarrow} M^{\vee}$ and $g(\rho)$ generates $M^{\vee}=\Hom(M,\mu_p)$ as an abelian group. Thus, $\zeta=g(\rho)(f(\sigma))$ generates $\mu_p$ as an abelian group.
\end{proof}

Combining the results of Lemmas \ref{conjbykers} and \ref{conjbysigma}, we see that $d=\sum_{i=0}^{p-1}{a_i e_{\rho^i}}$ commutes with $e_t$ for all $t\in\Gal(L/K_s^{\ker(g_0)})$ if and only if 
\begin{itemize}
\item $\sigma(a_i)=\zeta^{-i}a_i\ \quad\forall\ 0\leq i\leq p-1$, and 
\item if $t\in (\ker(f_0)\cap\ker(g_0))/N$ 
is such that $t\rho t^{-1}=\rho^{\ell}$, then $$t(a_i)=a_{\ell i}\ \quad\forall\ 0\leq i\leq p-1.$$
\end{itemize}

\begin{proposition}
\label{a_i}
Let $\alpha\in K_s^{\ker(f_0)}$ be such that $K_s^{\ker(f_0)}=K(\alpha)$. For each $0\leq i\leq p-1$, let $a_i=\sum_{j=0}^{p-1}{\zeta^{ij}\sigma^{j}(\alpha)}$. Then $d=\sum_{i=0}^{p-1}{a_i e_{\rho^i}}$ commutes with $e_t$ for all $t\in\Gal(L/K_s^{\ker(g_0)})$.
\end{proposition}

\begin{proof}
In order to show that $\sigma(a_i)=\zeta^{-i}a_i$, it suffices to show that $\sigma$ fixes $\zeta\in \mu_p$. But 
$\sigma$ acts trivially on both $M$ and $M^{\vee}=\Hom(M,\mu_p)$, so $\sigma$ must also act trivially on $\mu_p$. Now, let $t\in(\ker(f_0)\cap\ker(g_0))/N$ and suppose that $t\rho t^{-1}=\rho^{\ell}$. It suffices to show that $t(a_i)=a_{\ell i}$. We have
\begin{eqnarray*}
t(a_i)&=&\sum_{j=0}^{p-1}{t(\zeta)^{ij}t\sigma^{j}(\alpha)}
=\sum_{j=0}^{p-1}{t(\zeta)^{ij}(t\sigma t^{-1})^{j}t(\alpha)}\\
&=&\sum_{j=0}^{p-1}{t(\zeta)^{ij}(t\sigma t^{-1})^{j}(\alpha)}\ \quad\quad\textrm{since $t$ fixes $\alpha\in K_s^{\ker(f_0)}$.}\\
\end{eqnarray*}
Suppose that $t$ acts as multiplication by $k$ on $M$. Then $t$ acts as multiplication by $k\ell$ on $\mu_p$. We have isomorphisms of $G_{\Kk}$-modules $H_{\Mm}/{N_f}\cong M$ and $H_{\Mm^{\vee}}/{N_g}\cong M^{\vee}$ induced by $f$ and $g$ respectively. Hence, 
 \begin{eqnarray*}
 t(a_i)&=&\sum_{j=0}^{p-1}{t(\zeta)^{ij}(t\sigma t^{-1})^{j}(\alpha)}=\sum_{j=0}^{p-1}{\zeta^{ijk\ell}\sigma^{jk}(\alpha)}\\
 &=&\sum_{j=0}^{p-1}{t(\zeta)^{ij\ell}\sigma^{j}(\alpha)}= a_{\ell i}\\
 \end{eqnarray*}
as required.
\end{proof}

So we have demonstrated a candidate for $d$. It remains to check that this element is not in $K_s^{\ker(g_0)}$. It suffices to show that some $a_i$ with $i\geq 1$ is nonzero.

\begin{proposition}
\label{a_i nonzero}
Let $\alpha$ be such that $K_s^{\ker(f_0)}=K(\alpha)$ and $\Tr_{K_s^{\ker(f_0)}/K}(\alpha)=0$.
Then there exists $i\geq 1$ with $a_i=\sum_{j=0}^{p-1}{\zeta^{ij}\sigma^{j}(\alpha)}$ nonzero. Consequently, $d=\sum_{i=0}^{p-1}{a_i e_{\rho^i}}$ is not in $L$.
\end{proposition}

\begin{proof}
Let $V$ denote the Vandermonde matrix $(\zeta^{ij})_{0\leq i, j\leq p-1}$. Then $a_i$ is the $i$th row of $V(\alpha, \sigma(\alpha), \dots, \sigma^{p-1}(\alpha))^T$. Also, $$\det(V)=\prod_{0\leq i<j\leq p-1}{(\zeta^j-\zeta^i)}\neq 0.$$
Thus, $V(\alpha, \sigma(\alpha), \dots, \sigma^{p-1}(\alpha))^T$ is nonzero, so it has at least one nonzero row. In other words, at least one of the $a_i$'s is nonzero. But $$a_0=\alpha+\sigma(\alpha)+\cdots +\sigma^{p-1}(\alpha)=\Tr_{K_s^{\ker(f_0)}/K}(\alpha)=0.$$ Hence, there exists $i\geq 1$ with $a_i\neq 0$, as required.
\end{proof}

Since we assumed from the start that the characteristic of $K$ is not $p$, we can always arrange that $\Tr_{K_s^{\ker(f_0)}/K}(\alpha)=0$, by subtracting $p^{-1}\Tr_{K_s^{\ker(f_0)}/K}(\alpha)$ from $\alpha$. Thus, by Lemma \ref{generators for D} $\End_{A_{\varphi_0}}(\cS)^{\textrm{opp}}$ is generated as a $K$-algebra by $d$ together with the elements of $K_s^{\ker(g_0)}$.

\section{A minimal polynomial}
\label{min poly}

Our next aim is to show that the $K$-subalgebra of $\End_{A_{\varphi_0}}(\cS)^{\textrm{opp}}$ generated by $d$ is isomorphic to $K(\alpha)=K_s^{\ker(f_0)}$. We will do this by showing that $d$ and $p\alpha$ satisfy the same minimal polynomial over $K$. Let $\sigma$ be a generator for $(N_g\cap H_{\Mm})/N$ and let $\rho$ be a generator for $(N_f\cap H_{\Mm^{\vee}})/N$. Recall that $\alpha\in L$ is such that $K_s^{\ker(f_0)}=K(\alpha)$ and $\Tr_{K_s^{\ker(f_0)}/K}(\alpha)=0$. We have $d=\sum_{i=0}^{p-1}{a_i e_{\rho^i}}$, where $a_i=\sum_{j=0}^{p-1}{\zeta^{ij}\sigma^{j}(\alpha)}$ and $\zeta=\varphi_0(\sigma,\rho)=g(\rho)(f(\sigma))$. Similarly, let $\beta$ be such that $K_s^{\ker(g_0)}=K(\beta)$ and $\Tr_{K_s^{\ker(g_0)}/K}(\beta)=0.$ Let $b_i=\sum_{j=0}^{p-1}{\zeta^{ij}\rho^j(\beta)}$. In the proof of Proposition \ref{a_i nonzero}, we showed that $a_i\neq 0$ for some $i\geq 1$. The same argument shows that $b_m\neq 0$ for some $m\geq 1$. Choose such a $b_m$ and denote it by $\cB$. We would like to define a polynomial with roots $\cB^k d\cB^{-k}$ for $0\leq k\leq p-1$. We show that $\cB^k d\cB^{-k}$ commutes with $\cB^{\ell} d\cB^{-\ell}$ for every $k,\ell\in\bbZ$, so that $P(X)=\prod_{k=0}^{p-1}{(X-\cB^kd\cB^{-k})}$ is the desired polynomial. First, we prove two auxiliary lemmas.

\begin{lemma}
\label{B^kdB^-k}
For all $k\in\bbZ$, we have
\begin{equation*}
\cB^kd\cB^{-k}=\sum_{i=0}^{p-1}{\zeta^{ikm}a_i e_{\rho^i}}=\sum_{i=0}^{p-1}{\sigma^{-km}(a_i) e_{\rho^i}},
\end{equation*}
where $\cB=b_m=\sum_{j=0}^{p-1}{\zeta^{mj}\rho^j(\beta)}\neq 0$.
\end{lemma}

\begin{proof}
We have
\begin{eqnarray*}
\cB^kd\cB^{-k}&=&\cB^k \sum_{i=0}^{p-1}{a_i e_{\rho^i}} \cB^{-k}
= \sum_{i=0}^{p-1}{a_i \cB^k e_{\rho^i}\cB^{-k}}
=\sum_{i=0}^{p-1}{a_i \cB^k\rho^i(\cB^{-k})e_{\rho^i}}\\
&=&\sum_{i=0}^{p-1}{a_i \cB^k \zeta^{ikm}\cB^{-k}e_{\rho^i}}
=\sum_{i=0}^{p-1}{\zeta^{ikm}a_i e_{\rho^i}}
=\sum_{i=0}^{p-1}{\sigma^{-km}(a_i) e_{\rho^i}}.
\end{eqnarray*}
\end{proof}

\begin{lemma}
\label{powers of rho}
For all $i,j,k\in\bbZ$, we have $e_{\rho^i}\sigma^k(a_j) e_{\rho^j}=\sigma^k(a_j) e_{\rho^{i+j}}.$
\end{lemma}
\begin{proof}
Since $\rho\in (N_f\cap H_{\Mm^{\vee}})/N$, clearly $\rho$ fixes $\sigma^k(a_j)=\sum_{\ell=0}^{p-1}{\zeta^{j\ell}\sigma^{\ell+k}(\alpha)}$. Moreover, $f(\rho)=0$ and so \eqref{phi} gives $\varphi_0(\rho^i,\rho^j)=1$ for all $i,j\in\bbZ$. 
Hence, 
$$e_{\rho^i}\sigma^k(a_j) e_{\rho^j}=\rho^i\sigma^k(a_j)e_{\rho^i}e_{\rho^j}=\sigma^k(a_j) \varphi_0(\rho^i,\rho^j)e_{\rho^{i+j}}=\sigma^k(a_j) e_{\rho^{i+j}}.$$
\end{proof}

\begin{corollary}
\label{roots commute}
For all $k,\ell\in\bbZ$, $\cB^k d\cB^{-k}$ commutes with $\cB^{\ell} d\cB^{-\ell}$.
\end{corollary}

\begin{proof}
By Lemma \ref{B^kdB^-k}, we have
\begin{eqnarray*}
\cB^k d\cB^{-k}\cB^{\ell} d\cB^{-\ell}&=& \sum_{0\leq i,j\leq p-1}{\sigma^{-km}(a_i) e_{\rho^i}\sigma^{-\ell m}(a_j) e_{\rho^j}}.
\end{eqnarray*}
By Lemma \ref{powers of rho}, this is equal to
\[
\sum_{0\leq i,j\leq p-1}{\sigma^{-km}(a_i) \sigma^{-\ell m}(a_j) e_{\rho^{i+j}}}
=\cB^{\ell} d\cB^{-\ell}\cB^k d\cB^{-k}.
\]
\end{proof}

\begin{proposition}
\label{coeffs of P in K}
Let $P(X)=\prod_{k=0}^{p-1}{(X-\cB^kd\cB^{-k})}$. Then the coefficients of $P$ lie in $K$.
\end{proposition}

Since $K$ is the centre of $A_{\varphi_0}$, it suffices to show that the coefficients of $P$ commute with every element of $A_{\varphi_0}$. As a $K$-algebra, $A_{\varphi_0}$ is generated by the elements of $L$ and $\{e_s\}_{s\in\Gal(L/K)}$. We prove Proposition \ref{coeffs of P in K} in three steps.

\begin{lemma}
\label{coeffs commute with L}
The coefficients of $P$ commute with $x$ for every $x\in L$.
\end{lemma}

\begin{proof}
We know that $\rho$ generates $(N_f\cap H_{\Mm^{\vee}})/N\cong
H_{M^{\vee}}/N_g\cong M^{\vee}$, which has cardinality $p$. Therefore, $[L:L^{\langle\rho\rangle}]=p$ and $L=L^{\langle\rho\rangle}(x)$ for any $x\in L\setminus L^{\langle\rho\rangle}$. Since $\cB=b_m$ for some $1\leq m\leq p-1$, we have $\rho(\cB)=\zeta^{-m}\cB\neq\cB$. Therefore, $L=L^{\langle\rho\rangle}(\cB)$. Observe that conjugation by $\cB$ permutes the roots of $P$. For any $x\in L^{\langle\rho\rangle}$, we have $xdx^{-1}=d$, since $\rho^i(x)=x$ for such $x$. Hence, conjugation by $x\in L^{\langle\rho\rangle}$ fixes the roots of $P$. Therefore, conjugation by any element of $L$ fixes the coefficients of $P$.
\end{proof}

\begin{lemma}
\label{coeffs commute with Gal(L/L_g^v)}
The coefficients of $P$ commute with $e_t$ for all $t$ in $\Gal(L/K_s^{\ker(g_0)}).$
\end{lemma}

\begin{proof}
By construction, $d$ commutes with $e_t$ for all $t\in \Gal(L/K_s^{\ker(g_0)})$. Suppose $t\in \Gal(L/K_s^{\ker(g_0)})$ is such that $t$ acts as multiplication by $k$ on $M$ and $t$ acts as multiplication by $\ell$ on $M^{\vee}$. Then $t\sigma t^{-1}= \sigma^k$, because $f$ induces an isomorphism of $G_{\Kk}$-modules $H_{\Mm}/N_f\cong M$. Similarly, $t\rho t^{-1}= \rho^{\ell}$. By definition of the action on $M^{\vee}=\Hom(M,\mu_p)$, we have $t(\zeta)=\zeta^{k\ell}$. Therefore, 
\begin{eqnarray*}
e_t\cB e_t^{-1}&=&t(\cB)=\sum_{j=0}^{p-1}{t(\zeta)^{mj}(t\rho t^{-1})^jt(\beta)}=\sum_{j=0}^{p-1}{\zeta^{mjk\ell}\rho^{j\ell}t(\beta)}\\
&=&\sum_{j=0}^{p-1}{\zeta^{mjk}\rho^{j}t(\beta)}=\sum_{j=0}^{p-1}{\zeta^{mjk}\rho^{j}(\beta)}
\end{eqnarray*}
because $\beta\in K_s^{\ker(g_0)}$ and $t\in \Gal(L/K_s^{\ker(g_0)})$. Hence,
\begin{eqnarray*}
e_{t}\cB d\cB^{-1}e_t^{-1}&=&t(\cB)d t(\cB)^{-1}=t(\cB) \sum_{i=0}^{p-1}{a_i e_{\rho^i}}t(\cB)^{-1}
=\sum_{i=0}^{p-1}{t(\cB)a_i \rho^i(t(\cB))^{-1}e_{\rho^i}}\\
&=&\sum_{i=0}^{p-1}{t(\cB)(\zeta^{-ikm}t(\cB))^{-1}a_i  e_{\rho^i}}=\sum_{i=0}^{p-1}{\zeta^{ikm} a_i e_{\rho^i}}=\cB^k d\cB^{-k}
\end{eqnarray*}
by Lemma \ref{B^kdB^-k}. Thus, we see that conjugation by $e_t$ for $t\in\Gal(L/K_s^{\ker(g_0)})$ permutes the roots of $P$. Consequently, the coefficients of $P$ commute with $e_t$ for all $t\in\Gal(L/K_s^{\ker(g_0)})$.
\end{proof}

\begin{lemma}
\label{coeffs commute with e_s}
The coefficients of $P$ commute with $e_t$ for every $t$ in $\Gal(L/K).$
\end{lemma}

\begin{proof}
By Lemma \ref{coeffs commute with Gal(L/L_g^v)}, the coefficients of $P$ commute with $e_t$ for all $t\in\Gal(L/K_s^{\ker(g_0)})$. Thus, it suffices to prove that the coefficients of $P$ commute with $e_t$ for all $t$ in some set $R$ of left coset representatives for $\Gal(L/K_s^{\ker(g_0)})$ in $\Gal(L/K)$. By Lemma \ref{R gen by rho}, $R$ can be taken to be $\{\rho^i \}_{0\leq i\leq p-1}$. Since $f(\rho)=0$, \eqref{phi} gives $\varphi(\rho^i,t)=1$ for all $t\in G_{\Kk}/N$ and all $i\in\bbZ$. Hence, $e_{\rho^i}=e_{\rho}^i$ and it suffices to show that the coefficients of $P$ commute with $e_{\rho}$. By Lemma \ref{B^kdB^-k},
\begin{equation}
\label{eq:conjbyrho}
e_{\rho}\cB^kd\cB^{-k}e_{\rho}^{-1}=e_{\rho}\sum_{i=0}^{p-1}{\zeta^{ikm}a_i e_{\rho^i}}e_{\rho}^{-1}= \sum_{i=0}^{p-1}{\zeta^{ikm}e_{\rho}a_i e_{\rho^i}e_{\rho}^{-1}}
\end{equation}
because $\zeta$ is fixed by $\rho\in (N_f\cap H_{\Mm^{\vee}})/N$.
By Lemma \ref{powers of rho}, we have 
\begin{equation}
\label{eq:rho}
\sum_{i=0}^{p-1}{\zeta^{ikm}e_{\rho}a_i e_{\rho^i}e_{\rho}^{-1}}=\sum_{i=0}^{p-1}{\zeta^{ikm}a_i e_{\rho^{i+1}}e_{\rho}^{-1}}.
\end{equation}
We have $\rho\in (N_f\cap H_{\Mm^{\vee}})/N\subset\ker(f_0)/N$, whereby
for all $j\in\bbZ$ and all $s\in G_{\Kk}/N$ we have
$\varphi_0(\rho^j,s)=(\rho^j\cdot g_0(s))(f_0(\rho^j))=1.$
Hence, $e_{\rho^{-1}}=e_{\rho}^{-1}$ and $$e_{\rho^{i+1}}e_{\rho^{-1}}=\varphi_0(\rho^{i+1},\rho^{-1})e_{\rho^i}=e_{\rho^i}.$$
Therefore, equations \eqref{eq:conjbyrho} and \eqref{eq:rho} give $e_{\rho}\cB^kd\cB^{-k}e_{\rho}^{-1}=\cB^kd\cB^{-k}$ for all $0\leq k\leq p-1$. Hence, the coefficients of $P$ commute with $e_{\rho}$.

\end{proof}

Combining Lemma \ref{coeffs commute with L} and Lemma \ref{coeffs commute with e_s}, we see that the coefficients of $P$ lie in the centre of $A_{\varphi_0}$, which is $K$. Thus, we have proved Proposition \ref{coeffs of P in K}.

\begin{definition}
Let $Q(X)$ be the minimal polynomial of $p\alpha$ over $K$,
$Q(X)=\prod_{i=0}^{p-1}{(X-\sigma^i(p\alpha))}$. 
\end{definition}

 We will show that $P=Q$ and thus conclude that $P$ is irreducible and $K(d)\cong K(\alpha)$.



\begin{definition} 
We define $R(X,Y)=\prod_{k=0}^{p-1}{(X-\sum_{i=0}^{p-1}{\sigma^k(a_i)Y^i})}.$
\end{definition}

\begin{lemma}
We have
$P(X)=R(X,e_{\rho})$ and
$Q(X)=R(X,1)$.
\end{lemma}

\begin{proof}
Since $\rho\in (N_f\cap H_{\Mm^{\vee}})/N$, we have $f_0(\rho)=0$ and 
$\varphi_0(\rho^i,\rho^j)=1\quad \forall i,j\in \bbZ.$
Therefore, $e_{\rho}^i=e_{\rho^i}\quad \forall i\in \bbZ$. Thus, the equality $P(X)=R(X,e_{\rho})$ follows from Lemma \ref{B^kdB^-k}. Regarding the second claim, we have
\begin{equation*}
R(X,1)=\prod_{k=0}^{p-1}{(X-\sum_{i=0}^{p-1}{\sigma^k(a_i)})}=\prod_{k=0}^{p-1}{(X-\sigma^k\Bigl(\sum_{i=0}^{p-1}{a_i}\Bigr))}.
\end{equation*}
Observe that
\begin{equation*}
\sum_{i=0}^{p-1}{a_i}=\sum_{i=0}^{p-1}{\sum_{j=0}^{p-1}{\zeta^{ij}\sigma^j(\alpha)}}
=\sum_{j=0}^{p-1}{\sigma^j(\alpha)\sum_{i=0}^{p-1}{\zeta^{ij}}}
=p\alpha
\end{equation*}
because $\sum_{i=0}^{p-1}{\zeta^{ij}}=0$ unless $j=0$. This completes the proof that $R(X,1)=Q(X)$.
\end{proof}

\begin{proposition}
\label{P=Q}
We have $P(X)=Q(X).$
\end{proposition}

\begin{proof}
Write $R(X,Y)=\sum_{i=0}^{p-1}{\sum_{j=0}^{N}{c_{ij}X^iY^j}}$, where $N=(p-1)^2$ and $c_{ij}\in L$. Then $$R(X,Y)=\sum_{i=0}^{p-1}{X^i\sum_{k=0}^{p-1}{\sum_{\ \ \ \ \ j\equiv k\pmod{p}}{c_{ij}Y^j}}}$$ where the innermost sum runs over $0\leq j\leq N.$ Therefore,
\begin{eqnarray*}
P(X)= R(X, e_{\rho}) = \sum_{i=0}^{p-1}{X^i\sum_{k=0}^{p-1}{e_{\rho}^k\sum_{\ \ \ \ \ j\equiv k\pmod{p}}{c_{ij}}}}
\end{eqnarray*}
because $\rho$ has order $p$ in $\Gal(L/K)$, so $e_{\rho}^p=1$. Hence, the coefficient of $X^i$ is $\sum_{k=0}^{p-1}{e_{\rho}^k\sum_{j\equiv k\pmod{p}}{c_{ij}}}$. By Lemma \ref{coeffs of P in K}, the coefficients of $P$ lie in $K$. Therefore, 
$$\sum_{\ \ \ \ \ j\equiv k\pmod{p}}{c_{ij}}=0,\ \ \ \textrm{unless k=0.}$$
Whereby
\begin{equation}
\label{eq:R}
R(X,Y)=\sum_{i=0}^{p-1}{X^i\sum_{\ \ \ \ \ j\equiv 0\pmod{p}}{c_{ij}Y^j}}.
\end{equation}
Since $e_{\rho}^p=1$, \eqref{eq:R} gives $P(X)= R(X,e_{\rho})=R(X, 1)=Q(X).$
\end{proof}
\begin{corollary}
$P$ is the minimal polynomial of $d$ over $K$ and $K(d)\cong K(\alpha)=K_s^{\ker(f_0)}$.
\end{corollary}

\begin{proof} Proposition \ref{P=Q} shows that $d$ and $p\alpha$ are roots of the same polynomial over $K$. This polynomial is irreducible because it is the minimal polynomial of $p\alpha$. The characteristic of $K$ is not $p$, so $p$ is invertible and $K(d)\cong K(p\alpha)=K(\alpha)$.
\end{proof}

\section{The multiplication rule}
\label{multiplication}
Recall that $\alpha\in K_s$ is such that $K_s^{\ker(f_0)}=K(\alpha)$ and $\Tr_{K_s^{\ker(f_0)}/K}(\alpha)=0$. Similarly, $\beta\in K_s$ is such that $K_s^{\ker(g_0)}=K(\beta)$ and $\Tr_{K_s^{\ker(g_0)}/K}(\beta)=0$.

\begin{definition}
Let $$z=p^{-1}d=p^{-1}\sum_{i=0}^{p-1}{a_i e_{\rho^i}},$$
where $a_i=\sum_{j=0}^{p-1}{\zeta^{ij}\sigma^{j}(\alpha)}$. Thus, by Proposition \ref{P=Q}, the minimal polynomial of $z$ over $K$ is the same as that of $\alpha$.  
\end{definition}

Lemma \ref{generators for D} along with the work done in Section \ref{generators} tells us that $\End_{A_{\varphi_0}}(\cS)^{\textrm{opp}}$ is generated as a $K$-algebra by $\beta$ and $z$. The elements $\beta^iz^j$ for $0\leq i,j\leq p-1$ form a basis for $\End_{A_{\varphi_0}}(\cS)^{\textrm{opp}}$ as a $K$-vector space. To specify the multiplication on $\End_{A_{\varphi_0}}(\cS)^{\textrm{opp}}$, it is enough to specify structure constants $c_{ij}\in K$ such that 
$$z\beta=\sum_{0\leq i,j\leq p-1}{c_{ij}\beta^iz^j}.$$

\begin{lemma}
\label{z^j}
For all $0\leq j\leq p-1$, we have
$$z^j=p^{-1}\sum_{k=0}^{p-1}{h_{jk}e_{\rho^k}}$$
where $h_{jk}=\sum_{\ell=0}^{p-1}{\zeta^{k\ell}\sigma^{\ell}(\alpha^j)}\in L^{\langle\rho\rangle}=K_s^{N_f\cap H_{\Mm^{\vee}}}$.
\end{lemma}

\begin{proof}
It is easily seen that $h_{0k}= 0$ for all $1\leq k\leq p-1$ and $h_{00}=p$. Thus, the statement holds for $j=0$. The statement for $j=1$ follows immediately from the definition of $z$, upon observing that $h_{1k}=a_k$ for all $0\leq k\leq p-1$. We proceed by induction on $j$. Suppose that 
$$z^m=p^{-1}\sum_{k=0}^{p-1}{h_{mk}e_{\rho^k}}$$
for some $0\leq m\leq p-2$. Then,
\begin{eqnarray*}
z^{m+1}&=& z^m z
= \Bigl(p^{-1}\sum_{k=0}^{p-1}{h_{mk}e_{\rho^k}}\Bigr)\Bigl(p^{-1}\sum_{i=0}^{p-1}{a_i e_{\rho^i}}\Bigr)\\
&=&p^{-2}\sum_{i,k=0}^{p-1}{h_{mk} a_i e_{\rho^{k+i}}} \quad\quad\textrm{by Lemma \ref{powers of rho}}\\
&=&p^{-2}\sum_{n,k=0}^{p-1}{h_{mk} a_{n-k} e_{\rho^{n}}}.
\end{eqnarray*}



Hence, it suffices to prove that 
$$\sum_{k=0}^{p-1}{h_{mk}a_{n-k}}=ph_{(m+1)n}.$$
We have
\begin{equation*}
\sum_{k=0}^{p-1}{h_{mk}a_{n-k}}=\sum_{k,\ell,j=0}^{p-1}{\zeta^{k\ell}\sigma^{\ell}(\alpha^m)\zeta^{(n-k)j}\sigma^j(\alpha)}=\sum_{\ell,j=0}^{p-1}{\zeta^{nj}\sigma^{\ell}(\alpha^m)\sigma^j(\alpha)\sum_{k=0}^{p-1}{\zeta^{k(\ell-j)}}}.
\end{equation*}
Now observe that $\sum_{k=0}^{p-1}{\zeta^{k(\ell-j)}}$ equals zero when $\ell\neq j$, and equals $p$ when $\ell=j$. This concludes the proof.
\end{proof}

We want to find structure constants $c_{ij}\in K$ for $0\leq i,j\leq p-1$ such that 
\begin{equation}
\label{eq:structureconstants}
z\beta=\sum_{0\leq i,j\leq p-1}{c_{ij}\beta^i z^j}.
\end{equation}
By the definition of $z$,
\begin{equation}
\label{eq:zbeta}
z\beta = p^{-1}\sum_{i=0}^{p-1}{a_i e_{\rho^i}}\beta= p^{-1}\sum_{i=0}^{p-1}{a_i\rho^i(\beta) e_{\rho^i}}.
\end{equation}
Using Lemma \ref{z^j}, we expand the right-hand side of \eqref{eq:structureconstants} as
\begin{equation}
\label{eq:c_ij expression}
\sum_{i,j=0}^{p-1}{c_{ij}\beta^i z^j}= p^{-1}\sum_{i,j=0}^{p-1}{c_{ij}\beta^i \sum_{k=0}^{p-1}{h_{jk}e_{\rho^k}}}.
\end{equation} 
Hence, equating \eqref{eq:zbeta} and \eqref{eq:c_ij expression}, we obtain for every $0\leq k\leq p-1$
\begin{equation}
\label{eq:strcon2}
a_{k}\rho^{k}(\beta)=\sum_{i,j=0}^{p-1}{c_{ij}\beta^i {h_{jk}}}.
\end{equation}
Recall that $\ker(g_0)\cap H_{\Mm^{\vee}}=N_g\triangleleft G_{\Kk}$. So $K_s^{N_g}=K_s^{H_{\Mm^{\vee}}}(\beta)$ is Galois over $K$ and we write 
\begin{equation}
\label{eq:rhok}
\rho^{k}(\beta)=\sum_{i=0}^{p-1}{m_{i k}\beta^{i}}
\end{equation}
for $m_{i k}\in K_s^{H_{\Mm^{\vee}}}\subset L^{\langle\rho\rangle}$. We know that $L/L^{\langle\rho\rangle}$ has degree $p$ and is generated by $\beta$. Thus, the elements $1,\beta,\dots \beta^{p-1}$ form a basis for $L$ as a vector space over $L^{\langle\rho\rangle}$. Therefore, combining \eqref{eq:strcon2} and \eqref{eq:rhok} gives
\begin{equation}
\label{strcon3}
a_km_{ik}=\sum_{j=0}^{p-1}{c_{ij}h_{jk}}
\end{equation}
for all $0\leq i,k\leq p-1$.

\begin{definition}
\label{matrices}
We define three $p$-by-$p$ matrices $X$, $Y$ and $Z$. 
\begin{equation*}
X=(a_k m_{ik})_{i,k},\quad
Y=(c_{ik})_{i,k},\quad
Z=(h_{ik})_{i,k}.
\end{equation*}
In all three cases, the indices $i$ and $k$ run from $0$ to $p-1$. 
\end{definition} 
In terms of these matrices, \eqref{strcon3} becomes $X=YZ$, where $Y$ is to be found. We know that such a $Y$ exists and is unique because the elements $\beta^iz^j$ for $0\leq i,j\leq p-1$ form a basis for $\End_{A_{\varphi_0}}(\cS)^{\textrm{opp}}$. 

\begin{lemma}
The matrix $Z$ is invertible. Thus, $Y=XZ^{-1}$.
\end{lemma}

\begin{proof}
Suppose for contradiction that $Z$ is not invertible. Then $Z$ has a non-trivial kernel and there exists a nonzero matrix $T$ such that $TZ=0$. But then $(Y+T)Z=YZ=X$. This contradicts the fact that $Y$ is unique. 
\end{proof}

Therefore, $\End_{A_{\varphi_0}}(\cS)^{\textrm{opp}}$ has a basis $\{\beta^iz^j\}_{0\leq i,j\leq p-1}$ as a $K$-vector space, where $z$ satisfies the same minimal polynomial over $K$ as $\alpha$, and the multiplication satisfies
$$z\beta=\sum_{i,j=0}^{p-1}{c_{ij}\beta^iz^j}$$
where $(c_{ij})_{i,j}=XZ^{-1}$ for $X$ and $Z$ as defined in Definition \ref{matrices}. Thus, $\cD=\End_{A_{\varphi_0}}(\cS)^{\textrm{opp}}$ is as described in Theorem \ref{mainthm}. 

\section{An example}
We apply Theorem \ref{mainthm} to the case $K_s^{H_M}=K(\mu_p)$. In this case, any $1$-cocycle $f_0$ which represents a non-trivial element $f\in H^1(G_{\Kk},M)$ has $K_s^{\ker(f_0)}=K(\alpha)$ where $\alpha^p\in K$. Since $K_s^{H_M}=K(\mu_p)$, the action of $G_{\Kk}$ on $M^{\vee}$ is trivial. Thus, $H^1(G_{\Kk},M^{\vee})=\Hom(G_{\Kk},M^{\vee})$ and any non-trivial $g\in \Hom(G_{\Kk},M^{\vee})$ corresponds to a degree $p$ Galois extension $K_s^{\ker(g)}/K$. Let $K_s^{\ker(g)}=K_s(\beta)$ with $\Tr_{K(\beta)/K}(\beta)=0$. Let $\sigma\in G_{\Kk}$ be such that $\sigma$ fixes $K(\beta,\mu_p)$ and $\sigma(\alpha)/\alpha=\zeta_p$ for some primitive $p$th root of unity $\zeta_p$. Choose $\rho\in G_{\Kk}$ such that $\rho$ fixes $K(\alpha,\mu_p)$ and $(g(\rho))(f_0(\sigma))=\zeta_p$. We calculate 
\begin{eqnarray*}
h_{ij}=\sum_{\ell=0}^{p-1}{\zeta^{j\ell}\sigma^{\ell}(\alpha^i)}=\sum_{\ell=0}^{p-1}{\zeta^{(i+j)\ell}\alpha^i}.
\end{eqnarray*}
Hence, $h_{ij}=0$ unless $i+j\equiv 0\pmod{p}$. Write $\rho^j(\beta)=\sum_{i=0}^{p-1}{m_{ij}\beta^i}$ for $m_{ij}\in K$. An easy matrix calculation shows that
\begin{eqnarray*}
(h_{1j}m_{ij})_{i,j}(h_{ij})_{i,j}^{-1}=\bordermatrix{
                & \cr
                & 0  &  m_{0(p-1)} &0& \ldots & 0\cr
                & 0  &  m_{1(p-1)} &0& \ldots & 0\cr
                & \vdots & \vdots & \ddots & \vdots\cr
                & 0  &   m_{(p-1)(p-1)} &0  &\ldots & 0}.
\end{eqnarray*}
Now Theorem \ref{mainthm} tells us that the class of $f\cup g$ in $\Br(K)$ is given by the algebra $\cD$ with $K$-basis $\{\beta^iz^j\}_{0\leq i,j\leq p-1}$, where $z^p=\alpha^p\in K$, and we have 
\[z\beta z^{-1}=\sum_{i=0}^{p-1}{m_{i(p-1)}\beta^i}=\rho^{-1}(\beta).\]
So in this case $\cD$ is a cyclic algebra of dimension $p^2$ over $K$.

\paragraph{Acknowledgements}
I would like to thank my Ph.D. supervisor Tim Dokchitser for suggesting this problem and for his excellent support and guidance throughout my Ph.D. I am grateful to Hendrik Lenstra for a very detailed discussion of my work and for many suggestions for its improvement, in particular for pointing out a gap in the proof of Lemma \ref{generators for D} and providing the rest of the argument. I would like to thank Vladimir Dokchitser, Charles Vial, Rishi Vyas and James Newton for many useful discussions. I was supported by an EPSRC scholarship for the duration of this research.

\end{document}